\documentclass[12pt]{article}

\usepackage{amsfonts,graphicx,amsmath,amssymb,amsthm,geometry,mathtools,authblk,mathrsfs,bbm,dsfont,hyperref,nicefrac,enumerate,comment,xcolor,soul,graphicx,caption,float,subcaption,cite}

\newtheorem{lemma}{Lemma}
\newtheorem{remark}[lemma]{Remark}
\newtheorem{setting}[lemma]{Setting}

\newtheorem{theorem}[lemma]{Theorem}

\begin{document}
	
	\title{Mean-field type discrete stochastic linear quadratic optimal control problems}
	
	\date{}

	\author{Arzu Ahmadova$^1$, Nazim I. Mahmudov$^2$
		\bigskip
		\\
		\small{$^1$Faculty of Mathematics, University of Duisburg-Essen, 45127, Essen,
		Germany,\\
		e-mail: $arzu.ahmadova@uni-due.de$\\
		$^2$Department of Mathematics, Eastern Mediterranean University, 99628, T.R. North Cyprus,\\
		e-mail: $nazim.mahmudov@emu.edu.tr$}
	
	\smallskip
	}
	
	\maketitle
	\begin{abstract}
		\noindent In this paper, we consider linear quadratic optimal control with mean-field type for discrete-time stochastic systems with state and control dependent noise. 
		An optimal control problem is studied for a linear mean-field stochastic differential equation with a quadratic cost functional. The coefficients and the weighting matrices in the cost functional are all assumed to be deterministic.	
	\end{abstract}

	\section{Introduction}

Stochastic linear quadratic (LQ) optimal control for discrete-time systems. For the discrete-time LQ control problems with control and/or state-dependent noise, there are some works in the literature. An early work \cite{ku} deals with a special case, whose systems are described by a difference equation in which both the system matrix and the control matrix are multiplied by white, possibly correlated, scalar random sequences. In another paper \cite{beghi} the optimal control law for the systems with only control-dependent noise is derived. It is worth noting that the state weight matrix is nonnegative and the control weight matrix is positive definite in both papers. There are also several works concerning LQ optimal control problems, see textbook \cite{Anderson} and related articles \cite{EliottLiNi,Yong,RamiZhou}.
	
	The theory of MF-SDEs goes back to Kac, who in (Kac 1956) presented a stochastic toy model for the kinetic Vlasov equation of plasma, leading to the so-called stochastic McKean-Vlasov differential equation. Since then, the study of the related issues and their applications has become a remarkable and serious endeavor of researchers in the field of applied probability and optimal stochastic control, including financial engineering. Note that the problem (MF-LQ) reduces to the classical stochastic LQ optimal control problem in the absence of the mean-field part. For relevant results and historical comments on this topic, the reader is referred to (Ait Rami et al. \cite{rami}, Chen and Yong 2000, Chen et al. ), among others.
	
	Most previous researchers mainly studied indefinite stochastic LQ problems without constraints. However, some constraints are considerable importance in many physical systems. The finite time
	indefinite stochastic LQ control with linear terminal state constraint was discussed \cite{huang}. It is a valuable research topic to generalize those results to the discrete-time systems.
	
	In this paper, we discuss the maximum principle for the optimal control of discrete-time systems described by mean-field stochastic difference equations. As far as we know, there are few results on such stochastic control problems. In fact, discrete-time control systems are of great value in practice. For example, digital control can be formulated as a discrete-time control problem in which the sampled data are obtained at discrete times. In a discrete-time system, the Riccati difference equation plays an important role in synthesizing the optimal control. 

	 We believe that our method can also be applied to more complicated discrete-time stochastic optimal control problems, for example, problems with delays, terminal constraint problems, and problems with neutral term.

\section{Mathematical description}\label{secmath}
In Section \ref{secmath} we present in Setting \ref{setting} the mathematical framework which we use to study the discrete-time stochastic linear quadratic optimal control problems of mean-field type.
\begin{setting}\label{setting}
	Let $\left\| \cdot \right\|$ be a norm, $\langle \cdot, \cdot \rangle$ be an inner product, let $n_1, n_2 \in \mathbb{N}$ and denote the space of ($n_1\times n_2$)-matrices by $\mathbb{R}^{n_1\times n_2}$, and let $\mathbb{R}^{n_1}\coloneqq \mathbb{R}^{n_1\times 1}$, that is, each element of $\mathbb{R}^{n_1}$ is understood as a column vector, let $I$ be the unit matrix with appropriate dimension. For each matrix $A$, $A^{\intercal}$ denotes the transpose of $A$. For a vector $x\in \mathbb{R}^{n}$ denote by $x^{\intercal}$ its transpose. For a symmetric matrix $A$ and vectors $y,y_1,y_2$ of matching dimensions, we denote $A[y]^2\coloneqq y^{\intercal}Ay$, $A[y_1,y_2]\coloneqq y_1^{\intercal}Ay_2$.\\
	Let $(\Omega, \mathfrak{F}, \mathbb{P})$ be a complete probability space and $N$ be a positive integer. $\mathbb{T}:=\left\{  t_{k}=t_{0}+kh,\ h>0\right\}  _{k=0}^{N}$, let $\left\lbrace w(t_{k}): k=1,\ldots,N+1 \right\rbrace $ be a sequence of $\mathfrak{F}_{k}$-measurable $\mathbb{R}^{d}$-valued random variables, and let $\mathfrak{F}_{k}\subseteq \mathfrak{F}$ be the $\sigma$-field generated by $w(t_{1}),\ldots,w(t_{k})$, i.e., $\mathfrak{F}_{k}=\sigma\left\lbrace w(t_{1}),\ldots,w(t_{k})\right\rbrace $, $k=1,\ldots,N+1$, and $\mathfrak{F}_{0}=\left\lbrace \emptyset, \Omega \right\rbrace $. Let the expectation operator $\textbf{E}$ be denoted by $\textbf{E}x(t)=\int_{\Omega}x(t)\mathbb{P}(\mathrm{d}w)$ for each $w\in \Omega$. For each $t\geq  0$, $\textbf{E}\left\lbrace\cdot \mid \mathfrak{F}_{t} \right\rbrace  $ is the conditional expectation given by $\mathfrak{F}_{t}$. Assume for all $k\in \mathbb{N}$ that $ w_{h}(t_{k})\coloneqq w(t_{k+1})-w(t_{k})$ satisfies the following conditions:
	\begin{enumerate}
		\item[(wi)] For every $w_{h}\left(  t_{k}\right)  =\left(  w_{h}^{1}\left(
		t_{k}\right),...,w_{h}^{d}\left(  t_{k}\right)  \right)  ,\ w_{h}^{1}\left(
		t_{k}\right)  ,...,w_{h}^{d}\left(  t_{k}\right)  $ are independent
		$\mathbb{R}$-valued random variables.
		
		\item[(wii)] $\textbf{E}\left\{  w_{h}\left(  t_{k}\right)  \mid
		\mathfrak{F}_{k}\right\}  =0$, \quad $\textbf{E}\left\{  \left(  w_{h}^{j}\left(
		t_{k}\right)  \right)  ^{2}\mid\mathfrak{F}_{k}\right\}  =h,\\
		\textbf{E}\left(  w_{h}^{j}\left(  t_{k}\right)  \right)  ^{4}<\infty,$ \quad 
		$\textbf{E}\left( w_{h}^{m}\left(  t_{k}\right)  w_{h}^{l}\left(  t_{k}\right)\right) 
		=\left(  t_{k+1}-t_{k}\right)  \delta_{ml}I$.
	\end{enumerate}
	Moreover, let $\widehat{\mathfrak{F}}_{k}=\sigma\left\{  w_{h}\left(
	t_{k+1}\right),\ldots,w_{h}\left(  t_{N}\right)  \right\}$.  Note that $\mathfrak{F}_{k}$ and $\widehat{\mathfrak{F}}_{k}$ are
	independent. Let for all $v \in \mathbb{R}^{r}$ $\Delta f(t,v)\coloneqq f(t, \widehat{x}(t),\textbf{E}\widehat{x}(t),v)-f(t, \widehat{x}(t),\textbf{E}\widehat{x}(t),\widehat
	{u}(t))$,  and let $\mathbb{F}=\left\{  \mathfrak{F}_{k}:k=0,1,...,N\right\}$ be the set.
	A random variable $z=\left\{  z_{k}:k=0,1,...,N\right\}  $ is called
	$\mathbb{F}$-predictable if the random variable
	$z_{k}$ is $\mathfrak{F}_{k}$-measurable for every $k=0,1,...,N,$. Let $L^{2}(\Omega, \mathfrak{F}_{t_k}, \mathbb{R}^{n})$ be the set of all $\mathbb{R}^{n}$-valued $\mathfrak{F}_{t_k}$-measurable random variables $x(t_{k})$ with $\textbf{E}\|x(t_{k})\|^{2}<\infty$.
\end{setting}

	\section{Discrete-time stochastic LQ optimal control problem}\label{sec:6}
	
	In this section, we consider a class of stochastic discrete linear-quadratic optimal control problems. The state equation is the following linear stochastic difference equation:
	\begin{align}\label{lq}
		\begin{cases}
			x_{k+1}  =A_{k}x_{k}+\overline{A}_{k}\mathbf{E}x_{k}+B_{k}u_{k}+\left(C_{k}x_{k}+\overline{C}_{k}\mathbf{E}x_{k}+D_{k}u_{k}\right)  w_{k+1},\\
			x_{0}  =\xi\in\mathbb{R}^{n},
		\end{cases}
	\end{align}
	where $A_{k},B_{k}\in\mathbb{R}^{n\times n}$, and $C_{k},D_{k}\in
	\mathbb{R}^{n\times r}$ are given deterministic matrices. $w_{k},$
	$k=1,...,N,$ defined on a probability space $\left(  \Omega,\mathfrak{F}%
	,\mathbb{P}\right)  $, represents the stochastic disturbances, which is
	assumed to be a one dimensional martingale difference sequence satisfying (wi)
	and (wii). The cost functional associated with (\ref{lq}) is%
	\begin{equation}
		J\left(  u\right)  =\textbf{E}x_{N}^{\intercal}Q_{N}x_{N}+\textbf{E}\sum
		_{k=0}^{N-1}\left(  x_{k}^{\intercal}Q_{k}x_{k}+\left(\mathbf{E}x_{k} \right)^{\intercal} \overline
		{Q}_{k}\mathbf{E}x_{k}+u_{k}^{\intercal}R_{k}u_{k}\right)  , \label{lq1}%
	\end{equation}
	where $Q_{k}\in S^{n},R_{k}\in S^{r}$ are symmetric matrices with appropriate
	dimensions. Note that $S^{m}$
	is the set of all symmetric matrices of order $(m\times m)$. We now introduce the
	following assumption concerning the weighting matrices  $Q_{k},R_{k}, \overline{Q}_{k}, \overline{R}_{k}$ in the cost functional.
	\begin{description} 
		\item[(J)] The matrices $Q_{k}, \overline{Q}_{k}\in S^{n} $ and  $R_{k},\overline{R}_{k}\in S^{r}$ satisfy the following
		\begin{equation*}
			Q_{k}, 	Q_{k}+\overline{Q}_{k}\geq 0, \qquad R_{k}, R_{k}+\overline{R}_{k}>0.
		\end{equation*}
	\end{description}
	
	\textbf{Problem (MF-LQ):} The problem is to find $\widehat{u},\ \widehat{u}_{k}\in \mathcal{L}^{2}\left(
	\Omega,\mathfrak{F}_{k},\mathbb{R}^{r}\right)  $ such that
	\begin{equation*}
		J\left(  \widehat{u}\right)  =\inf\left\{  J\left( \widehat{u}_{k}\right)  :\widehat{u}_{k}\in
		\mathcal{L}^{2}\left(  \Omega,\mathfrak{F}_{k},\mathbb{R}^{r}\right)  \right\}  .
	\end{equation*}
	We then call $\widehat{u}$ an optimal control for stochastic discrete
	linear-quadratic problem \eqref{lq}-\eqref{lq1}.
	
	\begin{theorem}\label{lqp}
		Assume that assumption \textrm{(J)} holds. There exists a unique optimal control
		$\widehat{u}$ of the form%
		\[
		\widehat{u}_{k}=-\left(  R_{k}+B_{k}^{\intercal}P_{k+1}B_{k}+D_{k}^{\intercal
		}P_{k+1}D_{k}\right)  ^{-1}\left(  B_{k}^{\intercal}P_{k+1}A_{k}
		+D_{k}^{\intercal}P_{k+1}C_{k}\right)  x_{k}
		\]
		with
		\begin{align}\label{ricc}
			\begin{cases}
				P_{k}  &  =Q_{k}+A_{k}^{\intercal}P_{k+1}A_{k}+C_{k}^{\intercal}P_{k+1}%
				C_{k}\\
				&  +\left(  A_{k}^{\intercal}P_{k+1}B_{k}+C_{k}^{\intercal}P_{k+1}%
				D_{k}\right)  \left(  R_{k}+B_{k}^{\intercal}P_{k+1}B_{k}+D_{k}^{\intercal
				}P_{k+1}D_{k}\right)  ^{-1}\\
				&  \times\left(  B_{k}^{\intercal}P_{k+1}A_{k}+D_{k}^{\intercal}P_{k+1}%
				C_{k}\right)  ,\\
				P_{N}  &  =Q_{N}, 
			\end{cases}
		\end{align}
		having the property
		\[
		P_{k}\geq0.
		\]
		
	\end{theorem}
	
	\begin{proof}
		Assume that $u_{k},$ $k=0,1,\ldots,N-1,$ satisfies the following condition:
		\[
		\frac{\partial}{\partial u_{k}}H_{k}=\mathbf{E}\left\{  p_{k+1}^{\intercal
		}\mid\mathfrak{F}_{k}\right\}  B_{k}+\mathbf{E}\left\{  p_{k+1}^{\intercal
		}w_{k+1}^{j}\mid\mathfrak{F}_{k}\right\}  D_{k}-u_{k}^{\intercal}R_{k}=0,
		\]
		where
		\begin{align*}
			\begin{cases*}
				p_{k} =A_{k}^{\intercal}\mathbf{E}\left\{  p_{k+1}\mid\mathfrak{F}_{k}\right\}+\mathbf{E}\left\{  \overline{A}_{k}^{\intercal}p_{k+1}\right\}
				+\sum_{j=1}^{d}\left(  B_{k}^{j}\right)  ^{\intercal}\mathbf{E}\left\{  p_{k+1}w_{k+1}^{j}\mid\mathfrak{F}_{k}\right\}\\
				\hspace{0.5cm}+\sum_{j=1}^{d}
				\left(  \overline{B}_{k}^{j}\right)  ^{\intercal}\mathbf{E}\left\{
				p_{k+1}w_{k+1}^{j}\right\}  -Q_{k}x_{k}-\mathbf{E}\overline{Q}_{k}x_{k},\\
				p_{N+1}  =-Q_{N+1}x_{N+1}-\overline{Q}_{N+1}\mathbf{E}x_{N+1}
				,\ \ k=0,1,...,N.
			\end{cases*}
		\end{align*}
		
		It follows that for $k=N$ we have
		\begin{align}\label{qa1}
			-\mathbf{E}\left\{  x_{N+1}^{\intercal}Q_{N+1}\mid\mathfrak{F}_{N}\right\}
			B_{N}&-\left(  \mathbf{E}x_{N+1}^{\intercal}\right)  \overline{Q}_{N+1}%
			B_{N}-\mathbf{E}\left\{  x_{N+1}^{\intercal}Q_{N+1}w_{N+1}\mid\mathfrak{F}%
			_{N}\right\}  D_{N}\\
			&-\mathbf{E}\left\{  \left(  \mathbf{E}x_{N+1}^{\intercal
			}\right)  \overline{Q}_{N+1}w_{N+1}\mid\mathfrak{F}_{N}\right\}  D_{N}%
			-u_{N}^{\intercal}R_{N}=0. \nonumber
		\end{align}
		Substituting the following expression
		\[
		x_{N+1}=A_{N}x_{N}+\overline{A}_{N}\mathbf{E}x_{N}+B_{N}u_{N}+\left(
		C_{N}x_{N}+\overline{C}_{N}\mathbf{E}x_{N}+D_{N}u_{N}\right)  w_{N+1}%
		\]
		into \eqref{qa1} yields
		\begin{align*}
			&-\left(  A_{N}x_{N}+\overline{A}_{N}\mathbf{E}x_{N}+B_{N}u_{N}\right)^{\intercal}Q_{N+1}B_{N}\\
			&-\left(  A_{N}\mathbf{E}x_{N}+\overline{A}_{N}\mathbf{E}x_{N}+B_{N}\mathbf{E}u_{N}\right)^{\intercal}\overline{Q}_{N+1}B_{N}\\
			&-\left(  C_{N}x_{N}+\overline{C}_{N}\mathbf{E}x_{N}+D_{N}u_{N}\right)
			^{\intercal}Q_{N+1}D_{N}-u_{N}^{\intercal}R_{N}=0.
		\end{align*}
		Finding $u_{N}$ from above expression gives us 
		\begin{align*}
			&\left(R_{N}+B_{N}^{\intercal}Q_{N+1}B_{N}+B_{N}^{\intercal}\overline{Q}_{N+1}B_{N}\mathbf{E}\left\{  \cdot\right\}  +D_{N}^{\intercal}Q_{N+1}D_{N}\right)  u_{N}\\
			&=-\Big(  B_{N}^{\intercal}Q_{N+1}A_{N}+B_{N}^{\intercal}Q_{N+1}\overline{A}_{N}\mathbf{E}\left\{  \cdot\right\}  +B_{N}^{\intercal}\overline{Q}_{N+1}A_{N}+B_{N}^{\intercal}\overline{Q}_{N+1}\overline{A}_{N}\mathbf{E}
			\left\{  \cdot\right\}\\
			&+D_{N}^{\intercal}Q_{N+1}C_{N}+D_{N}^{\intercal
			}Q_{N+1}\overline{C}_{N}\mathbf{E}\left\{  \cdot\right\}\Big) x_{N}.
		\end{align*}
		Therefore, $u_{N}$ will become as follows:
		\begin{align*}
			&u_{N}=\widehat{u}_{N}=-\left(  R_{N}+B_{N}^{\intercal}Q_{N+1}B_{N}+D_{N}^{\intercal
			}Q_{N+1}D_{N}\right)^{-1}\Big(B_{N}^{\intercal}Q_{N+1}A_{N}+B_{N}^{\intercal}Q_{N+1}
			\overline{A}_{N}\mathbf{E}\left\{  \cdot\right\} \\
			& +D_{N}^{\intercal}
			Q_{N+1}C_{N}+D_{N}^{\intercal}Q_{N+1}\overline{C}_{N}\mathbf{E}\left\{
			\cdot\right\}  \Big)x_{N}=:L_{N}x_{N}.
		\end{align*}
		Similarly, we have
		\begin{align*}
			&\mathbf{E}\left\{  p_{N+1}\mid\mathfrak{F}_{N}\right\} \\
			&  =-\mathbf{E}%
			\left\{  Q_{N+1}x_{N+1}\mid\mathfrak{F}_{N}\right\} \\
			&  =-\mathbf{E}\left\{  Q_{N+1}\left(  A_{N}x_{N}+\overline{A}_{N}%
			\mathbf{E}x_{N}+B_{N}u_{N}+\left(  C_{N}x_{N}+\overline{C}_{N}\mathbf{E}%
			x_{N}+D_{N}u_{N}\right)  w_{N+1}\right)  \mid\mathfrak{F}_{N}\right\} \\
			&  =-\mathbf{E}\left\{  Q_{N+1}\left(  A_{N}x_{N}+\overline{A}_{N}%
			\mathbf{E}x_{N}+B_{N}L_{N}x_{N}\right)  \mid\mathfrak{F}_{N}\right\}
		\end{align*}
		and 
		\begin{align*}
			p_{N}  &  =A_{N}^{\intercal}\mathbf{E}\left\{  p_{N+1}\mid\mathfrak{F}%
			_{N}\right\}  +\overline{A}_{N}^{\intercal}\mathbf{E}\left\{  p_{N+1}\right\}\\
			&+C_{N}^{\intercal}\mathbf{E}\left\{  p_{N+1}w_{N+1}\mid\mathfrak{F}%
			_{N}\right\}  +\overline{C}_{N}^{\intercal}\mathbf{E}\left\{  p_{N+1}%
			w_{N+1}\right\}  -Q_{N}x_{N}\\
			&  =-\left(
			\begin{array}
				[c]{c}%
				A_{N}^{\intercal}Q_{N+1}A_{N}+A_{N}^{\intercal}Q_{N+1}\overline{A}%
				_{N}\mathbf{E}\left\lbrace \cdot\right\rbrace+A_{N}^{\intercal}Q_{N+1}B_{N}L_{N}\\
				+\overline{A}_{N}^{\intercal}Q_{N+1}A_{N}\mathbf{E}\left\lbrace \cdot\right\rbrace+\overline{A}%
				_{N}^{\intercal}Q_{N+1}\overline{A}_{N}\mathbf{E}\left\lbrace \cdot\right\rbrace+\overline{A}_{N}^{\intercal
				}Q_{N+1}B_{N}L_{N}\mathbf{E}\left\lbrace \cdot\right\rbrace\\
				+C_{N}^{\intercal}Q_{N+1}C_{N}+C_{N}^{\intercal}Q_{N+1}\overline{C}%
				_{N}\mathbf{E}\left\lbrace \cdot\right\rbrace+C_{N}^{\intercal}Q_{N+1}D_{N}L_{N}\\
				+\overline{C}_{N}^{\intercal}Q_{N+1}C_{N}+\overline{C}_{N}^{\intercal}%
				Q_{N+1}\overline{C}_{N}\mathbf{E}\left\lbrace \cdot\right\rbrace +\overline{C}_{N}^{\intercal}Q_{N+1}%
				D_{N}L_{N}+Q_{N}%
			\end{array}
			\right)  x_{N}.
		\end{align*}
		Therefore, we obtain
		\begin{align*}
			p_{N}  &  =A_{N}^{\intercal}Q_{N+1}A_{N}+A_{N}^{\intercal}Q_{N+1}\overline
			{A}_{N}\mathbf{E}\left\lbrace \cdot\right\rbrace+\overline{A}_{N}^{\intercal}Q_{N+1}A_{N}\mathbf{E}\left\lbrace \cdot\right\rbrace%
			+\overline{A}_{N}^{\intercal}Q_{N+1}\overline{A}_{N}\mathbf{E}\left\lbrace \cdot\right\rbrace\\
			&  +C_{N}^{\intercal}Q_{N+1}C_{N}+C_{N}^{\intercal}Q_{N+1}\overline{C}%
			_{N}\mathbf{E}\left\lbrace \cdot\right\rbrace+\overline{C}_{N}^{\intercal}Q_{N+1}C_{N}+\overline{C}%
			_{N}^{\intercal}Q_{N+1}\overline{C}_{N}\mathbf{E}\left\lbrace \cdot\right\rbrace\\
			&  +\left(  A_{N}^{\intercal}Q_{N+1}B_{N}+\overline{A}_{N}^{\intercal}%
			Q_{N+1}B_{N}+C_{N}^{\intercal}Q_{N+1}D_{N}+\overline{C}_{N}^{\intercal}%
			Q_{N+1}D_{N}\right) \\
			&\times \left(  R_{N}+B_{N}^{\intercal}Q_{N+1}B_{N}%
			+D_{N}^{\intercal}Q_{N+1}D_{N}\right)  ^{-1}\\
			&  \times\left(  B_{N}^{\intercal}Q_{N+1}A_{N}+B_{N}^{\intercal}%
			Q_{N+1}\overline{A}_{N}\mathbf{E}\left\{  \cdot\right\}  +D_{N}^{\intercal
			}Q_{N+1}C_{N}+D_{N}^{\intercal}Q_{N+1}\overline{C}_{N}\mathbf{E}\left\{
			\cdot\right\}  \right)  +Q_{N}x_{N}\\
			&  :=P_{N}x_{N}%
		\end{align*}
		Repeating this procedure step by step for $p_{N-2},p_{N-3},...,p_{0},$ yields
		the Riccati difference equation \eqref{ricc}.
	\end{proof}
	\begin{remark}
		In comparison to our method which we employed in this section, Eliotte et al. \cite{EliottLiNi} have taken expectation from stochastic difference equation in the beginning of the proof to get Riccati difference equation. Instead, we use the iterative method by taking derivative from Hamiltonian function with respect to $u_{k}$ to obtain \eqref{ricc}.
	\end{remark}

	\section{Conclusions}
	As a conclusion, we studied linear quadratic optimal control with mean-field type for discrete-time stochastic systems with state and control dependent noise. Moreover, our optimal control problem is studied for a linear mean-field stochastic difference equation with a quadratic cost functional. The coefficients and the weighting matrices in the cost functional are all assumed to be deterministic.
		
	Although there are many articles on the maximum principle of stochastic and deterministic systems, there remain many other interesting open problems concerning their fractional analogues, which can be extended by methods analogous to those used for fractional derivations of Caputo and Riemann-Liouville type. For example, one can consider the method given in \cite{yusubov-mahmudov} to study linear quadratic optimal control problem in which a dynamical system is controlled by a nonlinear Caputo fractional state equation.

\end{document}